\documentclass[a4paper,11pt]{article}

\usepackage{amsmath}
\usepackage{amsmath, amsthm, amssymb, amsfonts}
\usepackage{amscd}

\usepackage[french]{babel}

\usepackage{eucal}
\usepackage[latin1]{inputenc}
\usepackage[T1]{fontenc}
\usepackage{a4wide}
\usepackage{enumerate}
\usepackage{pstricks}
\usepackage{mathrsfs}
\usepackage{eufrak}
\usepackage{makeidx}

\everymath{\displaystyle}

\theoremstyle{definition}

\newtheorem{defn}{Définition}[section]

\theoremstyle{theorem}

\newtheorem{thm}[defn]{Théoreme}
\newtheorem{prop}[defn]{Proposition}
 \newtheorem{cor}[defn]{Corollaire}
 \newtheorem{lem}[defn]{Lemme}
 \newtheorem*{conj}{Conjecture de la masse positive}

 \renewcommand{\S}{\mathbb{S}}

 \newcommand{\R}{\mathbb{R}}

\newcommand{\D}{\mathscr{D}}

 \newcommand{\F}{\mathscr{F}}

 \renewcommand{\geq}{\geqslant}

\DeclareMathAlphabet{\mathpzc}{OT1}{pzc}{m}{it}

\def\Ric{{\mathrm{Ric}}}
\def\Rd{{\mathrm{R}}}
\def\tr{{\mathrm{tr}}}

\def\Id{{\mathrm{Id}}}
\def\X{{\mathcal{X}}}

\def\Qr{{\mathcal{Q}}}

\newcommand{\tld}[1]{\widetilde{#1}}

\usepackage[all]{xy}

\begin{document}

\title{Structures de Weyl ALF}
\author{Guillaume Vassal}
\date{\today}

\maketitle

\begin{abstract}
Dans cet article, nous définissons les notions de connexion de Weyl ALF et de masse conforme associée. Nous démontrons un théorème de la masse positive pour les structures de Weyl ALF.
\vskip .6cm


\medskip
\noindent{\it Mots clés :} Structure conforme, connexion de Weyl, variété asymptotiquement plate, variété ALF, théorème de la masse positive.
\end{abstract}

\section*{Introduction}

En géométrie, différents invariants géométriques, appelés \emph{masse}, sont associées à des variétés différentielles non compactes asymptotiques, dans un certain sens, à des variétés modèles à l'infini. Le cas le plus classique est celui des variétés asymptotiquement plates dont le modèle à l'infini est l'espace euclidien. Une variété riemannienne $(M,g)$ est asymptotiquement plate (à un seul bout) s'il existe un compact $K$ de $M$ tel que $M\setminus K$ est difféomorphe à l'extérieur d'une boule de $\R^{n}$ et tel que la métrique $g$ est asymptotique, dans un certain sens, à la métrique euclidienne de $\R^n$ sur l'ouvert $M\setminus K$. Sous certaines conditions, un invariant géométrique, calculé à l'infini, peut être associé à la variété asymptotiquement plate $(M,g)$. Cet invariant est la \emph{masse} de $(M,g)$ et lorsqu'il est bien défini, son expression est la suivante : 
\begin{align}
m(g)=\lim_{r\rightarrow\infty}\int_{S_{r}}\sum_{i,j=1}^{n}(\partial_{i}g_{ij}-\partial_{j}g_{ii})\partial_{j}\lrcorner dz,\nonumber
\end{align}
où $S_{r}$ est la sphère standard de rayon $r$ de $\R^{n}$. La conjecture de la masse positive est la suivante : 
\begin{conj}
Soit $(M,g)$ une variété asymptotiquement plate de dimension $n\geq 3$. Supposons que la masse de $(M,g)$ est bien définie. Si la courbure scalaire de $g$ est positive, la masse est positive. De plus, la masse est nulle si et seulement si $(M,g)$ est isométrique à l'espace euclidien.
\end{conj}
Ce problème trouve ses origines dans la physique. En effet, pour certains exemples, comme la métrique de Schwarzschild de $\R^4$ (voir \cite{lp}), la masse est un paramètre qui correspond à la masse des physiciens, et il est naturel que cette quantité soit positive dans de bonnes conditions. Récemment, J. Lohkamp a annoncé une preuve de la conjecture dans \cite{loh}. Jusqu'à maintenant, nous n'avions que des preuves partielles dont la preuve de E. Witten pour les variétés spinorielles de dimension quelconque \cite{wit}. Pour des informations plus précises sur ce sujet, le lecteur pourra consulter \cite{bar}, \cite{lp} ou \cite{park}. 

La notion de masse a été aussi introduite pour d'autres types de variétés non compactes. Un théorème de la masse positive est démontré pour des variétés asymptotiquement hyperboliques par P. Chru\`sciel et M. Herzlich \cite{ch} et pour des variétés asymptotiquement hyperboliques complexes par V. Minerbe et D. Maerten \cite{mm}. V. Minerbe démontre également un théorème de la masse positive pour les variétés ALF (Asymptotically Locally Flat) dans \cite{min}.

Dans cet article, nous allons nous intéresser particulièrement au travaux de V. Minerbe \cite{min}. Une variété riemannienne $(M,g)$ est ALF s'il existe un compact $K$ de $M$ tel que $M\setminus K$ est difféomorphe à l'espace total $\X$ d'une fibration en cercle au-dessus de $\R^{m}\setminus B_R$, où $B_R$ est la boule standard de rayon $R$, et tel que la métrique $g$ est asymptotique à une métrique modèle $h$ sur $\X$. L'espace total $\X$ muni de la métrique $h$ est le modèle à l'infini. L'exemple le plus simple est le produit $\R^{n}\times\S^1$ muni de la métrique produit. Dans ce cadre, la masse d'une variété ALF est une forme quadratique positive lorsque la courbure de Ricci de la variété est positive. 

Récemment, dans \cite{vas}, nous avons étendue la notion de variété asymptotiquement plate au cas des variétés conformes. En géométrie conforme, le rôle de la connexion de Levi-Civita en géométrie riemannienne est joué par l'espace affine des \emph{connexions de Weyl} qui sont des connexions sans torsion préservant la classe conforme de la variété. En particulier, pour chaque métrique $g$ de la classe conforme, la connexion de Levi-Civita de $g$, notée $\nabla^g$, est une connexion de Weyl et toute connexion de Weyl $D$ sur la variété conforme $(M,c)$ s'écrit sous la forme suivante : 
\begin{align}
D_{Y}X=\nabla^{g}_{Y}X +\theta_{g}(Y)X+\theta_{g}(X)Y-g(X,Y)\theta_{g}^{\sharp},\nonumber
\end{align}
où $\theta_{g}$ est la \emph{$1$-forme de Lee} de $D$ relativement à la métrique $g$ et $\theta_{g}^{\sharp}$ est le dual riemannien de $\theta_g$ relativement à $g$. La donnée $(M,c,D)$, où $D$ est une connexion de Weyl sur la variété conforme $(M,c)$, est appelée \emph{structure de Weyl}. Pour plus de détails concernant les structures de Weyl nous renvoyons le lecteur intéressé à \cite{cal}, \cite{cal2} et \cite{gau}.

Une structure de Weyl $(M,c,D)$ est \emph{asymptotiquement plate} s'il existe une métrique $g$ dans la classe conforme $c$ telle que $(M,g)$ est asymptotiquement plate au sens précédent et si la $1$-forme de Lee de $D$ relativement à $g$ satisfait certaines hypothèses de décroissance à l'infini. La métrique $g$ est alors appelée \emph{métrique adaptée} pour $(M,c,D)$. Nous  définissons la \emph{masse conforme}, notée $m(D)$, d'une structure de Weyl asymptotiquement plate $(M,c,D)$ évaluée en une métrique adaptée $g$ par la formule suivante : 
\begin{align}\label{m1}
m(D)(g)=m(g)+2(n-1)\int_{M}\delta^{g}(\theta_{g})v_{g},
\end{align}
où $\delta^{g}$ est la divergence relative à $g$, $v_g$ la forme volume sur $M$ définie par $g$ et $m(g)$ la masse riemannienne de $g$ définie précédemment (les conditions de décroissance à l'infini sur $\theta_g$ impliquent la convergence de l'intégrale). Nous démontrons en fait que la masse conforme ne dépend pas de la métrique adaptée choisie. Le théorème de la masse positive conforme suivant lui est associé : 
\begin{thm}\label{v1}\emph{(\cite{vas} Theorem 2.4.4)}
Soit $(M,c,D)$ une structure de Weyl asymptotiquement plate. Si la courbure scalaire de $D$ est positive, la masse conforme $m(D)$ est positive. Le masse est nulle si et seulement si $(M,c)$ est isomorphe à l'espace $\R^n$ muni de sa structure conforme canonique.
\end{thm}

Le but de cet article est d'étendre la théorie des structures de Weyl asymptotiquement plates développée dans \cite{vas} au cas des variétés ALF. La formule de Bochner joue un rôle central dans la construction de la masse des variétés ALF et dans la démonstration du théorème de la masse positive associé. 

Dans la première partie de cet article, nous présentons la théorie des connexions de Weyl et nous démontrons la formule de Bochner conforme. Nous rappelons ensuite, dans la seconde partie, les points essentielles concernant le théorème de la masse positive pour les variétés ALF. Enfin, dans la troisième partie, nous définissons la notion de structure de Weyl ALF et démontrons le théorème de la masse positive conforme associé.

Ce travail constitue une partie de la thèse de l'auteur sous la direction de Paul Gauduchon et Andrei Moroianu. Je les remercie chaleureusement tous deux pour leur soutien et leurs encouragements. Ce travail est également soutenu par l'ANR ACG (Aspects Conformes de la Géométrie) récemment constituée.

\section{Formule de Bochner conforme}

Nous présentons dans un premier temps les objets de géométrie conforme dont nous aurons besoin. Puis dans un second temps, nous démontrons la formule de Bochner conforme sous deux formes différentes.

\subsection{Structures de Weyl}

Soit $M$ une variété différentiable orientée de dimension $n$. Rappelons que sur toute variété il existe une famille de fibrés en droite réelles, notés $L^{k}$ avec $k$ réel, définis par 
$L^{k}=GL(M)\times _{|\det|^{k/n}}\R$, où $Gl(M)$ est le fibré des repères de $TM$. Les sections du fibré $L^k$ sont les \emph{densités de poids $k$}. Ces fibrés sont orientables donc triviaux. Nous définissons également le fibré des densités positives, noté $L^{k}_{+}$, par
$L_{+}^{k}=GL(M)\times _{|\det|^{k/n}}\R^{>0}$. Notons $L$ le fibré des densités de poids $1$ et remarquons que, lorsque $k$ est un entier positif, $L^{k}$ est le produit tensoriel de $k$ copies du fibré $L$. 

Nous allons maintenant définir la notion de connexion de Weyl sur une variété conforme. Soit $(M,c)$ une variété conforme. La structure conforme $c$ sur $M$ peut être vue comme une section normalisée du fibré $L^{-2}\otimes S^{2}(T^{\ast}M)$ telle que $c(X,X)\in L^{2}_{+}$, pour tout vecteur non nul $X$. La section $c$ est normalisée lorsque $\Lambda^{n}c$ est l'isomorphisme naturel entre $\Lambda^{n}T^{\ast}M\otimes\Lambda^{n}T^{\ast}M$ et $L^{-2n}$. Nous avons une correspondance biunivoque entre les métriques dans la classe conforme $c$ de $M$ est les sections positives du fibré $L$ : la métrique $g$ relative à une section $l$ de $L$ est donnée par $g=l^{-2}c$. La donnée d'une métrique $g$ dans $c$ trivialise le fibré $L$. Une \emph{connexion de Weyl} sur $(M,c)$ est une connexion linéaire sur $L$. Soit $D$ une connexion linéaire sur $L$. Cette connexion induit une connexion sans torsion sur le fibré tangent $TM$ de $M$. Nous notons également $D$ la connexion sans torsion sur $TM$ induite par la connexion linéaire $D$. De façon équivalente, une connexion de Weyl sur $(M,c)$ est une connexion sans torsion sur $TM$ préservant la structure conforme. Le fait que $D$ préserve la structure conforme signifie que $Dc=0$, où $D$ agit en tant que connexion sur $L$ et $TM$. Nous dirons que la donné $(M,c,D)$, où $D$ est une connexion de Weyl sur $(M,c)$, est une \emph{structure de Weyl}.

 Soient $D$ et $D^{\prime}$ deux connexions de Weyl sur  $(M,c)$. L'espace des connexions de Weyl sur $(M,c)$ est un espace affine modelé sur les $1$-formes réelles sur $M$. En effet, la différence entre deux connexions de Weyl $D$ et $D^{\prime}$ définit une $1$-forme sur $M$ à valeurs dans $\mathrm{End}(TM)$. Cependant, le fibré $L$ est trivial et par conséquent $\theta=D-D^{\prime}$ est une $1$-forme sur $M$ à valeurs réelles. La relation $D=D^{\prime}+\theta$ s'étent au fibré $TM$ de la façon suivante \cite{vas} : 
\begin{align}\label{r1}
D_{Y}X=D^{\prime}_{Y}X +\theta(Y)X+\theta(X)Y-c(X,Y)\theta^{\sharp},
\end{align}
où $\sharp$ et $\flat$ sont les isomorphismes musicaux définis par la structure conforme $c$. En particulier, pour une métrique $g$ dans la classe conforme $c$, la connexion de Levi-Civita de $g$ est une connexion de Weyl. Notons $\nabla^g$ la connexion de Levi-Civita de $g$ et $\theta_g$ la $1$-forme satisfaisant $D-\nabla^g=\theta_g$ sur $L$. La $1$-forme $\theta_g$ est la \emph{$1$-forme de Lee} de $D$ relative à $g$. 

La courbure de $D$, en tant que connexion linéaire sur $L$, est une $2$-forme réelle sur $M$ notée $F^D$ et appelée \emph{courbure de Faraday}. La courbure de Faraday est une $2$-forme fermée et particulier, pour toute métrique $g$ dans la classe conforme $c$, nous avons $F^D=d\theta_g$. 

La \emph{courbure de Weyl} de $D$, notée $\Rd^D$, est la courbure de $D$ considérée comme connexion sur $TM$ et définie par :
\begin{align}
\Rd^{D}_{X,Y}Z=D_{X}D_{Y}Z-D_{Y}D_{X}Z-D_{[X,Y]}Z,\nonumber
\end{align}
pour tout champs de vecteurs $X$, $Y$ et $Z$. Contrairement à la courbure riemannienne, le tenseur de courbure $\Rd^D$ n'est pas antisymétrique en tant qu'endomorphisme de $TM$. En effet, la courbure de Faraday est la partie symétrique de $\Rd^D$. Nous avons la décomposition suivante : 
\begin{align}\label{dcomp}
\Rd^{D}=\Rd^{D,a}+F^{D}\otimes \Id,
\end{align}
où $\Rd^{D,a}$ est la partie antisymétrique de la courbure de $D$ et où $\Id$ est l'identité des endomorphismes de $TM$. La \emph{courbure de Ricci} de la connexion de Weyl $D$ est donnée par : 
\begin{align}
\Ric^{D}(X,Y)=\mathrm{trace}(Z\mapsto\Rd^{D}_{Z,X}Y).\nonumber
\end{align}
L'\emph{opérateur de Ricci} de la connexion de Weyl, encore noté $\Ric^{D}$, est l'application linéaire de $TM$ dans $TM\otimes L^{-2}$ définie par :
\begin{align}
c(\Ric^{D}(X),Y)=\Ric^{D}(X,Y).\nonumber
\end{align}
Dans la base $c$-orthonormée $\{e_{i}\}_{i=1\ldots n}$, nous avons l'écriture locale suivante : 
\begin{align}\label{ric}
\Ric^{D}(X)=\sum_{i=1}^{n}(\Rd^{D,a}_{X,e_{i}}e_{i})l^{-2},
\end{align}
où $l$ est la section de $L$ associée à la base $c$-orthonormée $\{e_{i}\}_{i=1\ldots n}$. La courbure $\Ric^D$ est une section du fibré $T^{\ast}M\otimes T^{\ast}M$. La \emph{courbure scalaire} de $D$, notée $Scal^D$, est définie par :
\[Scal^D=\tr_{c}(\Ric^D).\]
Ainsi, la courbure scalaire de $D$ est une densité de poids $2$. Pour plus d'informations, le lecteur intéressé pourra consulter \cite{cal}, \cite{cal2}, \cite{gau}, \cite{vas} et \cite{weyl}.

\subsection{Formule de Bochner conforme}

Pour une variété riemannienne $(M,g)$, la formule de Bochner (voir \cite{bes} page $56-58$ ) est donnée par : 
\begin{align}\label{f1b}
(\D^g)^{2}\alpha=\Delta^g\alpha +\Ric^g(\alpha),\qquad\forall\alpha\in T^{\ast}M,
\end{align}
où $\Delta^g$, $\Ric^g$ et $\D^g$ sont respectivement l'opérateur Laplacien, l'opérateur de Ricci et l'opérateur de Dirac relatifs à la métrique $g$. L'opérateur de Dirac agissant sur les formes est donné par $\D^g=\delta^g+d$, où $\delta^g$ est la divergence définie par $g$ et $d$ la différentielle extérieure sur $M$. Le Laplacien est défini par $\Delta^g=-\tr(D^g\circ D^{g})$, où $D^g$ est la connexion de Levi-Civita de $g$. Dans la suite de cette section, nous allons établir une version conforme de la formule de Bochner. Pour cela, nous commençons par définir les opérateurs conformes analogues à ceux intervenant dans le cas riemannien. 

Soit $(M,c,D)$ une structure de Weyl. Définissons l'opérateur de divergence conforme relatif à $D$, noté $\delta^D$, et la différentielle extérieure, notée $d^D$,  induite par la connexion sans torsion $D$. Dans une base $c$-orthonormée $\{e_i\}_{i=1\ldots n}$, les opérateurs $\delta^D$ : $L^{k}\otimes\Lambda^{p}T^{\ast}M\rightarrow L^{k-2}\otimes\Lambda^{p-1}T^{\ast}M$ et $d^D$ : $L^{k}\otimes\Lambda^{p}T^{\ast}M\rightarrow L^{k}\otimes\Lambda^{p+1}T^{\ast}M$ sont définis par les formules suivantes : 
\begin{align}
\delta^D\omega=-\sum_{i=1}^{n}(e_{i}\lrcorner D_{e_i}\omega)l^{-2}\qquad\mbox{et}\qquad d^D\omega=\sum_{i=1}^{n}e_{i}^{\ast}\wedge D_{e_i}\omega,
\end{align}
où $\{e_{i}^{\ast}\}_{i=1\ldots n}$ est la base duale algébrique de $\{e_{i}\}_{i=1\ldots n}$ et $l$ la section de $L$ associée à la base $c$-orthonormée $\{e_{i}\}_{i=1\ldots n}$. Une section $\omega$ du fibré $L^{k}\otimes\Lambda^{p}M$ est aussi appelée \emph{$p$-forme de poids $k$}. Nous étudions maintenant le lien entre deux opérateurs différentiels ou deux opérateurs de divergence conforme associés à deux connexions de Weyl distinctes. Par la suite, nous aurons besoin de comparer les opérateurs $d^D$ et $\delta^D$ à leur correspondant riemanniens lorsque qu'une métrique sera fixée dans $c$.

\begin{lem}\label{lemA}
Soient $\tld{D}$ et $D$ deux connexions de Weyl sur $(M,c)$ telles que $\tld{D}=D+\theta$. Pour toute $p$-forme $\omega$ de poids $k$, c'est-à-dire $\omega\in C^{\infty}(\Lambda^{p}T^{\ast}M\otimes L^k)$, nous avons : 
\begin{align}\label{forA}
\tld{D}_{X}\omega=D_{X}\omega+(k-p)\theta(X)\omega-\theta\wedge(X\lrcorner\omega)+X^{\flat}\wedge(\theta^{\sharp}\lrcorner\omega),\qquad\forall X\in TM.
\end{align}
En particulier, si $\alpha\in C^{\infty}(T^{\ast}M\otimes L^k)$, nous avons : 
\begin{align}\label{forA2}
\tld{D}\alpha=D\alpha+(k-1)\theta\otimes\alpha-\alpha\otimes\theta+c(\alpha,\theta)c.
\end{align}
\end{lem}

\begin{proof}
Soit $\omega$ dans $C^{\infty}(\Lambda^{p}T^{\ast}M\otimes L^k)$. Les connexions de Weyl agissant sur les $p$-formes de poids $k$ sont reliées par la forme suivante : 
\begin{align}
\tld{D}_{X}\omega=D_{X}\omega+(k-p)\theta(X)\omega+d\nu(\theta\wedge X)(\omega),
\end{align}
où $\nu$ est la représentation définissant le fibré cotangent $T^{\ast}M$ comme fibré associé au fibré principal des repères de $TM$. Soit $\{X_{1},\ldots,X_{p}\}$ une famille de $p$ vecteurs de $TM$, nous avons : 
\begin{align}
d\nu(\theta\wedge X)(\omega)(X_{1},\ldots,X_{p})=&-\sum_{i=1}^{n}\omega(X_{1},\ldots,(\theta\wedge X)X_{i},\ldots,X_{n})\nonumber\\
=&-\sum_{i=1}^{p}\theta(X_i)\omega(X_{1},\ldots,X_{i-1},Y,\ldots,X_{p})\nonumber\\
&+\sum_{i=1}^{p}c(Y,X_{i})\omega(X_{1},\ldots,X_{i-1},\theta^{\sharp},\ldots,X_{p})\nonumber\\
=&-\sum_{i=1}^{p}(-1)^{i-1}\theta(X_i)\omega(Y,X_{1},\ldots,X_{p})\nonumber\\
&+\sum_{i=1}^{p}(-1)^{i-1}c(Y,X_{i})\omega(\theta^{\sharp},X_{1},\ldots,X_{p})\nonumber\\
=&-\theta\wedge(Y\lrcorner\omega)(X_{1},\ldots, X_{p})+Y^{\flat}\wedge(\theta^{\sharp}\lrcorner\omega)(X_{1},\ldots, X_{p}).\nonumber
\end{align}
Ce calcul nous donne bien la formule souhaitée : 
\[\tld{D}_{X}\omega=D_{X}\omega+(k-p)\theta(X)\omega-\theta\wedge(X\lrcorner\omega)+X^{\flat}\wedge(\theta^{\sharp}\lrcorner\omega).\]
\end{proof}

Nous en déduisons immédiatement le corollaire suivant : 
\begin{cor}\label{corB}
Soient $\tld{D}$ et $D$ deux connexions de Weyl sur $(M,c)$ telles que $\tld{D}=D+\theta$. Nous avons :
\begin{align}\label{r3}
d^{\tld{D}}\omega=d^{D}\omega+k\theta\wedge\omega,\qquad\forall\omega\in C^{\infty}(\Lambda^{p}T^{\ast}M\otimes L^{k}).
\end{align}
\end{cor}

De la même façon, nous pouvons montrer la proposition suivante : 

\begin{prop}
Soient $\tld{D}$ et $D$ deux connexions de Weyl sur $(M,c)$ telles que $\tld{D}=D+\theta$. Nous avons :
\begin{align}\label{delrel}
\delta^{\tld{D}}(\omega)=\delta^{D}(\omega)+(2-n-k+p)\theta^{\sharp}\lrcorner\omega,\qquad\forall\omega\in C^{\infty}(\Lambda^{p}T^{\ast}M\otimes L^{k}).
\end{align}
\end{prop}

Si $g$ est une métrique dans $c$, les opérateurs $\delta^g$ et $d$ peuvent être considérés comme les opérateurs de divergence conforme et comme la différentielle relative à la connexion de Levi-civita $\nabla^g$ de $g$ qui, en particulier, est une connexion de Weyl. Les formules (\ref{r3}) et (\ref{delrel}) précédentes nous permettent notamment de relier les opérateurs $\delta^D$ et $d^D$ aux opérateur $\delta^g$ et $d$ respectivement. Commençons maintenant la démonstration de la formule de Bochner conforme en établissant quelques formules préliminaires.

\begin{prop}\label{dcarre}
Pour toute connexion de Weyl $D$ sur $(M,c)$, nous avons la formule suivante :
\begin{align}\label{r4}
(d^D)^{2}\omega=kF^D\wedge\omega,\qquad\forall\omega\in C^{\infty}(\Lambda^{\ast}T^{\ast}M\otimes L^{k}).
\end{align}
\end{prop}

\begin{proof}
Soit $\nabla$ une connexion linéaire sur un fibré vectoriel $E$ sur $M$. Soit $d^{\nabla}$ l'opérateur différentielle associé à $\nabla$ agissant sur les formes sur $M$ à valeurs dans $E$. Pour toute forme $\psi$ à valeurs dans $E$, nous avons la formule suivante : 
\begin{align}\label{dnabla}
(d^{\nabla})^{2}\psi=\mathcal{R}^{\nabla}\wedge\psi,
\end{align}
où $\mathcal{R}^{\nabla}$ est la courbure de la connexion $\nabla$. Dans notre situation, la courbure de la connexion de Weyl $D$ agissant sur les formes de poids $k$ est $kF^{D}$. Pour toute section $\omega$ de $\Lambda^{\ast}T^{\ast}M\otimes L^{k}$, nous avons :  
\[(d^D)^{2}\omega=kF^D\wedge\omega.\]
\end{proof}

Nous allons maintenant établir la formule de Bochner conforme. Nous pouvons considérer l'opérateur $d^{D}+\delta^D$ comme un opérateur de Dirac sur l'espace des formes à poids, que nous notons $\D^D=d^{D}+\delta^D$. Soit $\alpha$ une section de $T^{\ast}M\otimes L^{k}$. En suivant la méthode pour établir la formule de Bochner (\ref{f1b}) (voir \cite{bes} page $56$) et les règles de calcul des connexions sans torsion, nous obtenons : 
\begin{align}\label{r5}
(\delta^{D}d^D +d^D\delta^D)\alpha=\Delta^{D}\alpha+\sum_{i=1}^{n}(\Rd^{D}_{e_{i},\cdot}\alpha)(e_i)l^{-2},
\end{align}
où $\Delta^D$ est le Laplacien relatif à $D$ défini par $\Delta^{D}=-\mathrm{tr}_{c}(D\circ D)$.
De plus, la courbure de la connexion de Weyl agit sur les $1$-formes de poids $k$ de la façon suivante : 
\begin{align}\label{r6}
(\Rd^{D}_{X,Y}\alpha)(Z)=kF^D(X,Y)\alpha(Z)-\alpha(\Rd^{D}_{X,Y}Z),\qquad\forall\alpha\in T^{\ast}M\otimes L^{k}.
\end{align}
Pour tout $X$ dant $TM$, les formules (\ref{r5}) et (\ref{r6}) nous donnent le calcul suivant :
\begin{align}\label{r7a}
((\delta^{D}d^D +d^{D}\delta^{D})\alpha)(X)=&
(\Delta^{D}\alpha)(X)+k\sum_{i=1}^{n}F^{D}(e_{i},X)\alpha(e_{i})l^{-2}-\alpha\big(\sum_{i=1}^{n}(\Rd^{D}_{e_{i},X}e_{i})l^{-2}\big),\nonumber\\
=&(\Delta^{D}\alpha)(X)+kF^{D}(\alpha^{\sharp},X)-\alpha(-\Ric^{D}(X))\nonumber\\
=&(\Delta^{D}\alpha)(X)+kF^{D}(\alpha^{\sharp},X)+c(\Ric^{D}(X),\alpha^{\sharp})\nonumber\\
=&(\Delta^{D}\alpha)(X)+kF^{D}(\alpha^{\sharp},X)+\Ric^{D}(X,\alpha^{\sharp}).
\end{align}
Ainsi, en contractant par $\alpha$ la formule obtenue par le calcul (\ref{r7a}) précédent, nous obtenons : 
\begin{align}\label{ok1}
\langle(\delta^{D}d^D +d^{D}\delta^{D})\alpha,\alpha\rangle=\langle\Delta^{D}\alpha,\alpha\rangle+\Ric^{D}(\alpha^{\sharp},\alpha^{\sharp}).
\end{align}
De plus, en utilisant la formule (\ref{r4}), nous avons $(\D^{D})^{2}\alpha=\delta^{D}(\delta^{D}\alpha)+(\delta^{D}d^D +d^{D}\delta^{D})\alpha+kF^{D}\wedge\alpha$. Il suffit alors de remarquer que $\delta^{D}(\delta^{D}\alpha)=0$ et que les formes $\alpha$ et $F^{D}\wedge\alpha$ n'ont pas le même degré pour en déduire la formule suivante : 
\begin{align}\label{ok2}
\langle(\D^{D})^{2}\alpha,\alpha\rangle=\langle(\delta^{D}d^D +d^{D}\delta^{D})\alpha,\alpha\rangle.
\end{align}
D'après les formules (\ref{ok1}) et (\ref{ok2}), nous obtenons la formule de Bochner conforme suivante : 

\begin{thm}\label{bc}
Pour toute connexion de Weyl $D$ sur $(M,c)$, nous avons : 
\begin{align}
\langle (\D^{D})^2 \alpha,\alpha\rangle=\langle \Delta^{D}(\alpha),\alpha\rangle-\Ric^{D}(\alpha^{\sharp},\alpha^{\sharp}), \qquad\forall\alpha\in C^{\infty}(T^{\ast}M\otimes L^{k}).
\end{align}
\end{thm}

Nous terminons cette section en établissant une version intégrale de la formule de Bochner conforme. Rappelons que les objets naturellement intégrables sur une variété conforme sont les densités de poids $-n$, c'est-à-dire les sections du fibré $L^{-n}$. Nous allons donc établir une formule de Bochner conforme mettant en jeux des sections de $L^{-n}$. Notons $\langle\ ,\,\rangle$ le produit scalaire conforme sur les $p$-formes de poids quelconque induit par $c$. 

\begin{prop}\label{boch2}
Pour toute connexion de Weyl $D$ sur $(M,c)$, nous avons la formule suivante : 
\begin{align}\label{eq6}
\langle D\alpha,D\alpha\rangle+\Ric^{D}(\alpha^{\sharp},\alpha^{\sharp})-\langle\D^{D}\alpha,\D^{D}\alpha\rangle=-\delta^{D}(\zeta_{\alpha}),
\end{align}
où $\zeta_{\alpha}(X)=\langle\alpha,D\alpha+\delta^{D}(\alpha)X^{\flat}-X\lrcorner d^{D}\alpha\rangle$, pour tout $X$ dans $TM$.
\end{prop}

\begin{proof}
Soient $\alpha$ et $\beta$ des $1$-formes de poids $k$. Remarquons tout d'abord que, compte tenu du degré des formes en présence, nous avons : 
\begin{align}
\langle (\D^{D})^2 \alpha,\beta\rangle=\langle(\delta^{D}d^{D}+d^{D}\delta^{D})\alpha,\beta\rangle.\nonumber
\end{align}
En suivant la même démonstration que pour la formule ($12$) page $512$ dans \cite{vas}, nous obtenons : 
\begin{align}\label{eqq1}
\langle D\alpha,D\beta\rangle=\langle\Delta^{D}(\beta),\alpha\rangle-\delta^{D}(\langle\alpha,D\beta\rangle),
\end{align}
De façon similaire, nous obtenons également les deux formules suivantes :
\begin{align}\label{eq2}
\langle d^{D}\alpha,d^{D}\beta\rangle=\langle\alpha,\delta^{D}d^{D}\beta\rangle+\delta^{D}(\zeta_{\alpha,\beta}^{1})\qquad\mbox{et}\qquad\langle \delta^{D}\alpha,\delta^{D}\beta\rangle=\langle\alpha,d^{D}\delta^{D}\beta\rangle+\delta^{D}(\zeta_{\alpha,\beta}^{2}),
\end{align}
où $\zeta^{1}_{\alpha,\beta}$ et  $\zeta^{2}_{\alpha,\beta}$ sont les $1$-formes de poids $2k-2$ définies par :
\begin{align}
 \zeta_{\alpha,\beta}^{1}(X)=\langle\alpha,\delta^{D}(\beta)X^{\flat}\rangle\qquad\mbox{et}\qquad \zeta^{2}_{\alpha,\beta}(X)=-\langle\alpha,X\lrcorner d^{D}\beta\rangle.\nonumber
\end{align}
Les formes $\alpha$ et $\beta$ ayant le même degré, nous avons : 
\begin{align}\label{eq4}
\langle\D^{D}\alpha,\D^{D}\beta\rangle=\langle\delta^{D}\alpha,\delta^{D}\beta\rangle+\langle d^{D}\alpha,d^{D}\beta\rangle.
\end{align}
Ainsi, en sommant les équations données par (\ref{eq2}), nous obtenons : 
\begin{align}\label{eq5}
\langle\D^{D}\alpha,\D^{D}\beta\rangle=\langle\alpha,(\D^D)^2\beta\rangle+\delta^{D}(\zeta_{\alpha,\beta}),
\end{align}
où $\zeta_{\alpha,\beta}(X)=\langle\alpha,\delta^{D}(\beta)X^{\flat}-X\lrcorner d^{D}\beta\rangle$. Posons $\alpha=\beta$. Le théorème \ref{bc} et les équations (\ref{eq5}) et (\ref{eqq1}) nous donnent immédiatement le résultat souhaité : 
\begin{align}
\langle D\alpha,D\alpha\rangle+\Ric^{D}(\alpha^{\sharp},\alpha^{\sharp})-\langle\D^{D}\alpha,\D^{D}\alpha\rangle=-\delta^{D}(\zeta_{\alpha}),\nonumber
\end{align}
où $\zeta_{\alpha}(X)=\langle\alpha,D\alpha+\delta^{D}(\alpha)X^{\flat}-X\lrcorner d^{D}\alpha\rangle$. 
\end{proof}
Nous remarquons par exemple que $D\alpha$ est une section du fibré $L^{k}\otimes T^{\ast}M\otimes T^{\ast}M$ et que $\zeta_{\alpha}$ est une $1$-forme de poids $2k-2$, par conséquent, $\langle D\alpha,D\alpha\rangle$ et $\delta^{D}(\zeta_{\alpha})$ sont des sections de $L^{2k-4}$. Les termes de l'équation (\ref{eq6}) sont donc des sections du fibré $L^{2k-4}$. Ces termes sont des densités d'intégration si et seulement si $k=(4-n)/2$. La proposition \ref{boch2} ci-dessus nous donne alors la seconde formule de Bochner conforme :

\begin{thm}\label{both2}
Soient $(M,c)$ une variété conforme orientée et $\Omega$ un compact de $M$. Pour toute connexion de Weyl $D$ sur $(M,c)$ et  pour toute $1$-forme $\alpha$ de poids $(4-n)/2$, nous avons la formule de Bochner conforme globale suivante : 
\begin{align}\label{bochner2}
\int_{\Omega}\big(|D\alpha|^2 +\Ric^{D}(\alpha^{\sharp},\alpha^{\sharp})-|\D^{D}\alpha|^2 \big)=-\int_{\Omega}\delta^{D}(\zeta_{\alpha}),
\end{align}
où $\zeta_{\alpha}(X)=\langle\alpha,D_{X}\alpha+\delta^{D}(\alpha)X^{\flat}-X\lrcorner d^{D}\alpha\rangle$ pour tous champs de vecteurs $X$ et, où $|D\alpha|^2 =\langle D\alpha,D\alpha\rangle$ est la norme conforme induite par $c$.
\end{thm}

\section{Structures conformes ALF}

\subsection{Variétés riemanniennes ALF}

 Nous allons rappeler le théorème de la masse positive dans le cas des variétés ALF démontré par V. Minerbe \cite{min}. Commençons par décrire de manière précise l'espace modèle à l'infini. Soit $\pi$ : $\X\rightarrow\R^{m}\setminus B_{R}$ une fibration en cercles de longueur constante $L$, où $B_{R}$ est la boule de $\R^{m}$ de rayon $R$. Soient $\check{x}_{i}$ les coordonnées canoniques sur $\R^{m}$. On note $x_{i}=\pi^{\ast}\check{x}_{i}$ les coordonnées induites sur $\X$; celles-ci définissent une distance $r=\sqrt{x_{1}^{2}+\cdots x_{m}^{2}}$ sur $\X$. Soit $S$ le champs de vecteurs sur $\X$ engendré par l'action de $\S^1$. On pose $T=\frac{L}{2\pi}S$. Une connexion $\eta$ sur $\X$ est une $1$-forme $\S^{1}$-invariante sur $\X$ telle que $\eta(T)=1$. Soit $\eta$ une connexion sur $\X$. La $2$-forme $d\eta$ est le pull-back d'une $2$-forme $\omega$ sur $\R^{m}$. Supposons que $d\eta=\pi^{\ast}\omega$ possède les propriétés de décroissance suivantes :
\begin{align}
\omega=O(r^{1-m})\qquad\mbox{et}\qquad d\omega=O(r^{-m})\nonumber
\end{align}
Nous définissons la métrique $h$ sur $\X$ par :
\begin{align}
h=\pi^{\ast}g_{\R^{m}}+\eta^{2}=dx^{2}+\eta^{2},\nonumber
\end{align}
où $g_{\R^m}$ est la métrique canonique sur $\R^m$. L'espace $(\X,h)$ est l'espace modèle à l'infini pour les variétés ALF. Il y a deux exemples simples de tels espaces : la fibration triviale, où $\X$ est le produit $\R^{m}\times\S^1$  muni de la métrique produit $h=dt^2+dx^2$, et la fibration de Hopf de $\R^{4}\setminus\{0\}$ sur $\R^{3}\setminus\{0\}$ munie de la métrique $h=dx^{2}+\eta^{2}$, où $\eta$ est la forme de contact standard de $\S^{3}$ et $dx^2$ le pull-back de la métrique standard de $\R^3$.
\begin{defn}
Soit $(M,g)$ une variété riemannienne complète orientée de dimension $m+1$, avec $m\geq3$. La variété riemannienne $(M,g)$ est ALF s'il existe un compact $K$ de $M$ tel que $M\setminus K$ est difféomorphe à $\X$, où $(\X,h)$ est l'espace modèle décrit ci-dessus, et tel que la métrique $g$ vérifie sur $\X$ les estimations suivantes : 
\begin{align}
g=h+O(r^{2-m}),\qquad\nabla^{h}g=O(r^{1-m})\qquad\mbox{et}\qquad\nabla^{h,2}g=O(r^{-m}),\nonumber
\end{align}
où $\nabla^{h}$ est la connexion de Levi-Civita de la métrique $h$.
\end{defn}

Soit $(M,g)$ une variété ALF de dimension $m+1$ asymptotique à $(\X,h)$ à l'infini. Notons que $(dx_{1},\cdots,dx_{m},\eta)$ est une base $h$-orthonormée du fibré cotangent de $\X$ et soit $(X_{1},\cdots,X_{m},T)$ sa base duale. Notons $\mathscr{Z}$ l'espace des champs de vecteurs engendré par $\{X_{1},\cdots,X_{m}\}$. Dans le cas des variétés asymptotiquement plates, la masse est une nombre réel, en revanche, pour une variété ALF la masse est une forme quadratique définie sur l'espace $\mathscr{Z}$. 

\begin{defn}\label{malf}(V. Minerbe, \cite{min} Definition $2$ page $944$)
La masse de $(M,g)$ est la forme quadratique $\mathcal{Q}_{g}$ définie sur $\mathscr{Z}$ par : 
\begin{align}
\mathcal{Q}_{g}(Z)=\frac{1}{\omega_{n}L}\limsup_{r\rightarrow\infty}\int_{\partial B_{r}}\ast_{h}q_{g,h}(Z),
\end{align}
pour tout champs de vecteurs $Z$ dans $\mathscr{Z}$. Le volume de la sphère standard $\S^{m-1}$ est noté $\omega_{n}$, $\ast_h$ est l'opérateur de Hodge relatif à $h$ et la quantité $q_{g,h}(Z)$ est donnée par :
\begin{align}\label{def2}
q_{g,h}(Z)=-(\mathrm{div}_{h}g)(Z)\tld{\alpha}_{Z}-\frac{1}{2}\big(d(\tr_{h}g)(Z)\tld{\alpha}_{Z}+d(g(Z,Z))\big),
\end{align}
où $\tld{\alpha}_{Z}$ est la $1$-forme duale de $Z$ relativement à la métrique $h$.
\end{defn}

Le théorème de la masse positive correspondant est le suivant : 

\begin{thm}\emph{(V. Minerbe, \cite{min} Theorem 3 page 944).}
Soit $(M,g)$ une variété riemannienne complète, orientée, et de dimension $m+1$, avec $m\geq 3$. Supposons que $(M,g)$ est ALF et de courbure de Ricci positif ou nulle. Alors la masse $\mathcal{Q}_{g,h}$ est une forme quadratique positive, et celle-ci est nulle si et seulement si $(M,g)$ est le produit standard $\R^{n}\times\S^{1}$.
\end{thm} 

Terminons cette section par un lemme central dans la théorie de la masse positive ALF et dont nous aurons besoin par la suite.

\begin{lem}\label{lemal}\emph{(V. Minerbe, \cite{min} Lemma 6 page 942)}
Soit $(M,g)$ une variété ALF. Soient $Z$ dans $\mathscr{Z}$ et $\tld{\alpha}_{Z}$ la forme duale de $Z$ relativement à la métrique $h$. Alors il existe une $1$-forme $\alpha_Z$ sur $M$ telle que $(d+\delta)\alpha_Z=0$ et satisfaisant les conditions suivantes : 
\begin{align}
\alpha_{Z}-\tld{\alpha}_{Z}=O(r^{2-m+\varepsilon})\qquad\mbox{et}\qquad r^{-2+m-2\varepsilon}\nabla^{g}(\alpha_{Z}-\tld{\alpha}_{Z})\in L^{1},
\end{align}
où $\nabla^{g}$ est la connexion de Levi-Civita de $g$.
\end{lem}

\subsection{Structures de Weyl ALF}

\section{Structures de Weyl ALF}
 
 Dans cette partie, nous appliquons la théorie des connexions de Weyl asymptotiquement plates \cite{vas} aux variétés ALF. Nous conservons les notations introduites précédemment.

\begin{defn}
La structure de Weyl $(M,c,D)$ est ALF de dimension $m+1$ s'il existe une métrique $g$ dans $c$ telle que la variété riemannienne $(M,g)$ est ALF et telle que la forme de Lee $\theta_{g}$ de $D$ relative à $g$ satisfait, sur $\X$, les conditions de décroissance suivantes :
\begin{align}
\theta_{g}=O(r^{1-m})\qquad\mbox{et}\qquad d\theta_{g}=O(r^{2-m}).
\end{align}
Une métrique $g$ satisfaisant les conditions précédentes est une \emph{métrique adaptée} pour la structure de Weyl ALF $(M,c,D)$.
\end{defn}

Cette définition étend la notion de variété ALF au cadre conforme. Nous allons définir la masse conforme associée à une structure de Weyl ALF. Comme pour la masse conforme d'une structure de Weyl asymptotiquement plate, voir \cite{vas}, la masse conforme d'une structure de Weyl ALF sera évaluée en une métrique adaptée $g$. Nous montrons alors que la masse est indépendante du choix d'une telle métrique. Notons $\langle\ ,\,\rangle_h$ et $|\ |_h$ respectivement le produit scalaire et la norme relatifs à la métrique $h$ sur $\X$ induits sur les formes de $M$.

\begin{defn}
Soient $(M,c,D)$ une structure de Weyl ALF de dimension $m+1$ et $g$ une métrique adaptée pour $(M,c,D)$. La masse conforme associée à $(M,c,D)$ évaluée en $g$, que nous notons $m^{D}_{h}(g)$, est la forme quadratique sur $\mathscr{Z}$ définie par :
\begin{align}
m^{D}_{h}(g)(Z)=\Qr_{g}(Z)+\frac{1}{\omega_{n}L}\limsup_{r\rightarrow\infty}\int_{\partial B_r}\ast_{h}\big((1-m)\langle\theta_{g},\tld{\alpha}_{Z}\rangle_{h}\tld{\alpha}_{Z}-|\tld{\alpha}_{Z}|_{h}^{2}\theta_{g}\big),\qquad\forall Z\in\mathscr{Z}.
\end{align}
La forme quadratique  $\Qr_{g}$ est la masse de la variété ALF $(M,g)$ donnée par la définition \ref{malf} et $\tld{\alpha}_{Z}$ est la forme duale de $Z$ relativement à $h$.
\end{defn}

Nous allons montrer que cette masse ne dépend pas du choix de la métrique adaptée choisie. Caractérisons l'espace des métriques adaptées dans $c$. Nous considérons l'espace de fonctions $\F$ défini par :
\begin{align}
\F=\{f\in C^{\infty}(M,]0,+\infty[)\mbox{ : }f-1=O(r^{2-m})\mbox{, }\partial_{k}f=O(r^{1-m})\mbox{ et }\partial_{l}\partial_{k}f=O(r^{-m})\}.
\end{align}
Nous avons la proposition suivante : 

\begin{prop}\label{gfg}
Soit $g$ une métrique adaptée pour la structure de Weyl ALF $(M,c,D)$. Considérons le changement conforme $\tld{g}=fg$. La métrique $\tld{g}$ est adaptée pour $(M,c,D)$ si et seulement si $f$ appartient à $\F$.
\end{prop}

\begin{proof}
D'après la définition des variétés ALF, sur $\X$, nous avons $g_{ij}=h_{ij}+a_{ij}$, où $a_{ij}=O(r^{2-m})$, $\partial_{k}a_{ij}=O(r^{1-m})$ et $\partial_{l}\partial_{k}a_{ij}=O(r^{-m})$. Nous écrivons $\tld{g}_{ij}=h_{ij}+(f-1)h_{ij}+fa_{ij}$. Nous posons $b_{ij}=(f-1)h_{ij}+fa_{ij}$. Si $f$ appartient à $\F$, nous avons $b_{ij}=O(r^{2-m})$, $\partial_{k}b_{ij}=O(r^{1-m})$ et $\partial_{l}\partial_{k}b_{ij}=O(r^{-m})$. Ainsi, si $f$ appartient à $\F$, la métrique $\tld{g}$ est ALF. Supposons que $\tld{g}$ est une métrique ALF. Nous avons alors $\tld{g}_{ij}=h_{ij}+b_{ij}$, où $b_{ij}=O(r^{2-m})$, $\partial_{k}b_{ij}=O(r^{1-m})$ et $\partial_{l}\partial_{k}b_{ij}=O(r^{-m})$. De plus, comme $\tld{g}=fg$, nous avons $h_{ij}+b_{ij}=f(h_{ij}+a_{ij})$. Ainsi, il est clair que $f$ est bornée. L'égalité $(f-1)h_{ij}=b_{ij}-fa_{ij}$ montre alors que $f-1=O(r^{2-m})$. De la même façon, nous montrons que $\partial_{k}f=O(r^{1-m})$ et $\partial_{l}\partial_{k}f=O(r^{-m})$. Donc, la fonction $f$ appartient à $\F$.
\end{proof}

Afin de démontrer l'invariance de la masse conforme de $(M,c,D)$, nous  établissons préalablement la formule de changement conforme de la masse $\Qr_{g}$ de la variété ALF $(M,g)$.

\begin{prop}\label{pcc}
Soit $\tld{g}=fg$, où le changement conforme $f$ est une fonction dans $\F$. Nous avons alors la formule suivante : 
\begin{align}
\Qr_{\tld{g}}(Z)=\Qr_{g}(Z)+\frac{1}{2\omega_{n}L}\limsup_{r\rightarrow\infty}\int_{\partial B_{r}}\ast_{h}\big((1-m)\langle df,\tld{\alpha}_{Z}\rangle_{h}\tld{\alpha}_{Z}-|\tld{\alpha}_{Z}|_{h}^{2}df\big),\qquad\forall Z\in\mathscr{Z},
\end{align}
où $\tld{\alpha}_{Z}$ est la forme duale de $Z$ relativement à $h$.
\end{prop}

\begin{proof}
Soit $Z$ dans $\mathscr{Z}$.  Réécrivons le terme $q_{\tld{g},h}(Z)$, défini par (\ref{def2}), dans la base $(X_{1},\ldots,X_{m},X_{m+1})$, où $X_{m+1}=T$. En utilisant la convention d'Einstein pour les indices redondants, nous avons :
\begin{align}\label{eqp21}
q_{\tld{g},h}(Z)=(\nabla^{h}_{X_{b}}\tld{g})(X_{b},Z)\tld{\alpha}_{Z}-\frac{1}{2}d(\tld{g}(X_{b},X_{b}))(Z)\tld{\alpha}_{Z}-\frac{1}{2}d(\tld{g}(Z,Z)).
\end{align} 
En exprimant le terme de droite de l'équation (\ref{eqp21}) en fonction de $f$ et de $g$, nous obtenons :
\begin{align}\label{eqp22}
q_{\tld{g},h}(Z)=fq_{g,h}(Z)+df(X_{b})g(X_{b},Z)\tld{\alpha}_{Z}-\frac{1}{2}df(Z)g(X_{b},X_{b})\tld{\alpha}_{Z}-\frac{1}{2}g(Z,Z)df.
\end{align}
Par hypothèse, nous avons $g=h+O(r^{2-m})$ et $df=O(r^{1-m})$. Par conséquent, l'équation (\ref{eqp22}) nous donne : 
\begin{align}
q_{\tld{g},h}(Z)=&fq_{g,h}(Z)+df(X_{b})h(X_{b},Z)\tld{\alpha}_{Z}-\frac{(m+1)}{2}df(Z)\tld{\alpha}_{Z}-\frac{1}{2}h(Z,Z)df+O(r^{3-2m})\nonumber\\
=&fq_{g,h}(Z)+df(Z)\tld{\alpha}_{Z}-\frac{(m+1)}{2}df(Z)\tld{\alpha}_{Z}-\frac{1}{2}|\tld{\alpha}_{Z}|^{2}_{h}df+O(r^{3-2m})\nonumber\\
=&fq_{g,h}(Z)-\frac{(m-1)}{2}df(Z)\tld{\alpha}_{Z}-\frac{1}{2}|\tld{\alpha}_{Z}|^{2}_{h}df+O(r^{3-2m}).\nonumber
\end{align}
Cependant, par définition de $\tld{\alpha}_{Z}$, $df(Z)=\langle df,\tld{\alpha}_{Z}\rangle_{h}$. De plus, par hypothèse, $f=1+O(r^{2-m})$ et $q_{g,h}(Z)=O(r^{1-m})$, donc $fq_{g,h}(Z)=q_{g,h}(Z)+O(r^{3-2m})$. Ainsi, d'après le calcul précédent, nous avons la formule suivante :
\begin{align}\label{eqp23}
q_{\tld{g},h}(Z)=q_{g,h}(Z)-\frac{1}{2}|\tld{\alpha}|_{h}^{2}df+\frac{(1-m)}{2}\langle df,\tld{\alpha}\rangle_{h}\tld{\alpha}+O(r^{3-2m}).
\end{align}
Enfin, en intégrant la formule (\ref{eqp23}) sur les sphères de rayon $R$ de $\R^{m+1}$ et en passant à la limite supérieure, nous obtenons la formule souhaitée : 
\begin{align}
\Qr_{\tld{g}}(Z)=\Qr_{g}(Z)+\frac{1}{2\omega_{n}L}\limsup_{r\rightarrow\infty}\int_{\partial B_{r}}\ast_{h}\big((1-m)\langle df,\tld{\alpha}\rangle_{h}\tld{\alpha}-|\tld{\alpha}|_{h}^{2}df\big).
\end{align}
\end{proof}

Soit $(M,c,D)$ une structure de Weyl ALF. Nous savons comment évolue la forme de Lee de $D$ relativement à deux métriques distinctes de $c$. Pour $\tld{g}$ et $g$ dans $c$ telles que $\tld{g}=fg$, nous avons : 
\begin{align}\label{theta}
\theta_{\tld{g}}=\theta_{g}-\frac{df}{2f}.
\end{align}
Nous sommes alors en mesure de démontrer que la masse conforme de $(M,c,D)$ ne dépend pas de la métrique adaptée choisie.

\begin{prop}
Soit $(M,c,D)$ une structure de Weyl ALF. Soient $g_1$ et $g_2$ deux métriques adaptées à $(M,c,D)$. Nous avons :
\begin{align}
m^{D}_{h}(g_1)=m^{D}_{h}(g_2).\nonumber
\end{align}
\end{prop}

\begin{proof}
D'après la proposition \ref{gfg}, il existe $f$ dans $\F$ telle que $g_2=fg_1$. La formule (\ref{theta}) et la proposition \ref{pcc} nous donnent alors : 
\begin{align}
m^{D}_{h}(g_2)=m^{D}_{h}(g_1)&+\frac{1}{2\omega_{n}L}\limsup_{r\rightarrow\infty}\int_{\partial B_{r}}\ast_{h}\big((1-m)\langle df,\tld{\alpha}_{Z}\rangle_{h}\tld{\alpha}_{Z}-|\tld{\alpha}_{Z}|_{h}^{2}df\big)\nonumber\\
&-\frac{1}{2\omega_{n}L}\limsup_{r\rightarrow\infty}\int_{\partial B_{r}}\frac{1}{f}\ast_{h}\big((1-m)\langle df,\tld{\alpha}_{Z}\rangle_{h}\tld{\alpha}_{Z}-|\tld{\alpha}_{Z}|_{h}^{2}df\big).\nonumber
\end{align}
De plus, d'après le lemme $2.1.13$ page $519$ de \cite{vas}, pour toute fonction $f$ dans $\F$ et pour toute $n$-forme $\omega$ sur $M$ telles que $\omega=O(r^{1-m})$, nous avons :
\begin{align}
\limsup_{r\rightarrow\infty}\int_{\partial B_{r}}\frac{1}{f}\omega=\limsup_{r\rightarrow\infty}\int_{\partial B_{r}}\omega,\nonumber
\end{align}
Par conséquent, les deux derniers termes du calcul précédent s'annulent et nous obtenons : 
\begin{align}
m^{D}_{h}(g_1)=m^{D}_{h}(g_2).\nonumber
\end{align}
\end{proof}
Il nous reste a énoncer et démontrer le théorème de la masse positive conforme pour les structures de Weyl ALF. 

\begin{thm}\label{mpcalf}
Soit $(M,c,D)$ une structure de Weyl ALF de dimension $m+1$. Supposons que la courbure de Ricci de $D$ est positive. Alors la masse $m^{D}_{h}$ de $(M,c,D)$ est une forme quadratique positive sur $\mathscr{Z}$.
\end{thm}

\begin{proof}
Soient $g$ une métrique adaptée à $(M,c,D)$ et $\nabla^g$ la connexion de Levi-Civita de $g$. Soient $Z$ dans $\mathscr{Z}$ et $\tld{\alpha}_{Z}$ la forme duale de $Z$ relativement à la métrique $h$. Considérons la $1$-forme $\alpha_{Z}$ associée à $\tld{\alpha}_{Z}$ donnée par le lemme \ref{lemal}. Posons $\alpha=l^{\frac{4-(m+1)}{2}}\alpha_{Z}=l^{\frac{3-m}{2}}\alpha_{Z}$, où $l$ est la section de $L$ correspondante à la métrique $g$. Soit $(M_{r})$ une suite strictement croissante de compacts de $M$ tels que $\partial M_{r}=\partial B_{r}$. D'après la formule de Bochner conforme \ref{bc}, nous avons :
\begin{align}\label{mp0}
\int_{M_r}\big(|D\alpha|^2 +\Ric^{D}(\alpha^{\sharp},\alpha^{\sharp})-|\D^{D}\alpha|^2 \big)=-\int_{M_r}\delta^{D}(\zeta_{\alpha}),\nonumber
\end{align}
où $\zeta_{\alpha}(X)=\langle\alpha,D_{X}\alpha+\delta^{D}(\alpha)X^{\flat}-X\lrcorner d^{D}\alpha\rangle$ pour tous champs de vecteurs $X$ sur $TM$. Démontrons que le membre de droite de cette égalité converge vers la masse conforme de $(M,c,D)$ lorsque $r$ tend vers l'infini. Nous commençons par exprimer la divergence conforme de $\zeta_{\alpha}$ relativement à la métrique $g$. Notons que : 
\begin{align}
\zeta_{\alpha}(X)=\langle\alpha,D_{X}\alpha\rangle+\delta^{D}(\alpha)\alpha(X)+d^{D}\alpha(\alpha^{\sharp},X).
\end{align}
D'après la formule (\ref{forA2}), nous avons : 
\begin{align}
D\alpha=\nabla^{g}\alpha+(k-1)\theta_{g}\otimes\alpha-\alpha\otimes\theta_{g}+c(\alpha,\theta_{g})c\nonumber
\end{align}
En contractant la formule ci-dessus par la $1$-forme $\alpha$ de poids $(3-m)/2$, nous obtenons : 
\begin{align}\label{mp1}
\langle\alpha,D_{X}\alpha\rangle=\langle\alpha,\nabla^{g}_{X}\alpha\rangle+\frac{1-m}{2}|\alpha|^{2}\theta_{g}.
\end{align} 
D'après la formule (\ref{r3}), nous avons :
\begin{align}\label{mp2a}
d^D\alpha=d\alpha+\frac{(3-m)}{2}\theta_{g}\wedge\alpha.
\end{align}
Nous en déduisons immédiatement la formule : 
\begin{align}\label{mp2b}
d^{D}\alpha(\alpha^{\sharp},\cdot)=d\alpha(\alpha^{\sharp},\cdot)+\frac{3-m}{2}\langle\theta_{g},\alpha\rangle\alpha-\frac{3-m}{2}|\alpha|^{2}\theta_{g}.
\end{align}
De la même façon que pour la formule (\ref{r3}), nous démontrons : 
\begin{align}\label{mp3}
\delta^{D}(\alpha)=\delta(\alpha)-\frac{(m+1)}{2}\langle\theta_{g},\alpha\rangle.
\end{align}
Enfin, en utilisant les formules (\ref{mp1}), (\ref{mp2b}) et (\ref{mp3}) pour calculer $\zeta_{\alpha}$, nous obtenons : 
\begin{align}\label{mp4}
\zeta_{\alpha}(X)=\zeta^{g}_{\alpha}(X)-|\alpha|^{2}\theta_{g}(X)+(1-m)\langle\theta_{g},\alpha\rangle\alpha(X),\qquad\forall X\in TM,
\end{align}
où $\zeta^{g}_{\alpha}(X)=\langle\alpha,\nabla_{X}^{g}\alpha+\delta(\alpha)X^{\flat}-X\lrcorner d\alpha\rangle$.
Rappelons que lorsqu'une métrique $g$ est fixée, le fibré $L$ est trivialisé par la section $l$ de $L$ correspondante à $g$. Nous avons alors l'identification suivante : 
\begin{align}\label{mp4i}
-\int_{M_r}\delta^{D}(\zeta_{\alpha})=-\int_{M_r}\delta(\zeta_{\alpha})v_g=\int_{\partial B_{r}}\ast_{g}(\zeta_{\alpha}),
\end{align}
où $v_g$ et $\ast_g$ sont respectivement la forme volume et l'opérateur de Hodge associés à $g$ et $\zeta_{\alpha}$ est identifié à une $1$-forme réelle sur $M$. La formule (\ref{mp4i}) et l'expression (\ref{mp4}) de $\zeta_\alpha$, nous donnent : 
\begin{align}\label{mp4ii}
 -\int_{M_r}\delta^{D}(\zeta_{\alpha})=\int_{\partial B_{r}}\ast_{g}(\zeta^{g}_{\alpha_{Z}})+\int_{\partial B_r}\ast_{g}\big((1-m)\langle\theta_{g},\alpha_{Z}\rangle_{g}\alpha_{Z}-|\alpha_{Z}|_{g}^{2}\theta_{g}\big),
\end{align}
où $\langle\, ,\ \rangle_g$ est le produit scalaire sur les formes induit par $g$. De plus, l'intégrale du terme $\ast_{g}(\zeta_{\alpha_{Z}})$ correspond au terme obtenu en intégrant la formule de Bochner riemannienne, et celui-ci converge vers la masse de la variété riemannienne ALF $(M,g)$. En effet, le lemme $8$ page $943$ de \cite{min} donne : 
\begin{align}
\limsup_{r\rightarrow\infty}\int_{\partial B_{r}}\ast_{g}(\zeta^{g}_{\alpha_{Z}})=\omega_{n}L\mathcal{Q}_{h}(Z).
\end{align}
Les hypothèses de décroissance de $\theta_g$, de $\alpha_{Z}$ et de la métrique $g$ nous donnent : 
\begin{align}
(1-m)\langle\theta_{g},\alpha_{Z}\rangle_{g}\alpha_{Z}-|\alpha_{Z}|_{g}^{2}\theta_{g}=(1-m)\langle\theta_{g},\tld{\alpha}_{Z}\rangle_{h}\tld{\alpha}_{Z}-|\tld{\alpha}_{Z}|_{h}^{2}\theta_{g}+O(r^{3-2m}).
\end{align}
Par conséquent, par passage à la limite supérieure dans l'équation (\ref{mp4ii}), nous obtenons la formule suivante :
\begin{align}
-\int_{M}\delta^{D}(\zeta_{\alpha})=\omega_{n}L\Big(\mathrm{Q}_{h}(Z)+\frac{1}{\omega_{n}L}\limsup_{r\rightarrow\infty}\int_{\partial B_r}\ast_{h}\big((1-m)\langle\theta_{g},\tld{\alpha}_{Z}\rangle_{h}\tld{\alpha}_{Z}-|\tld{\alpha}_{Z}|_{h}^{2}\theta_{g}\big)\Big).\nonumber
\end{align}
Nous voyons donc apparaître la masse  de la structure de Weyl ALF $(M,c,D)$ et nous avons : 
\begin{align}\label{mp5}
\int_{M}\big(|D\alpha|^2 +\Ric^{D}(\alpha^{\sharp},\alpha^{\sharp})-|\D^{D}\alpha|^2 \big)=\omega_{n}Lm^{D}_{h}(Z).
\end{align}
Il nous reste à montrer que $\int_{M}|\D^{D}\alpha|^2=0$. D'après les formules (\ref{mp2a}) et (\ref{mp3}), nous avons : 
\begin{align}\label{mp6beforea}
\D^{D}\alpha=\delta^{D}(\alpha)+d^{D}\alpha=(d+\delta)\alpha_{Z}-\frac{m+1}{2}\langle\theta_{g},\alpha_{Z}\rangle_{g}+\frac{3-m}{2}\theta_{g}\wedge\alpha_{Z}.
\end{align}
Par construction, nous avons $(d+\delta)\alpha_{Z}=0$. D'après les hypothèses de décroissance de $\theta_{g}$, nous avons donc $\langle\theta_{g},\alpha_{Z}\rangle_{g}=O(r^{2-2m})$ et $\theta_{g}\wedge\alpha_{Z}=O(r^{2-2m})$. Ainsi, nous obtenons : 
\begin{align}\label{mp6before}
\limsup_{r\rightarrow\infty}\int_{\partial B_r}|\langle\theta_{g},\alpha_{Z}\rangle_{g}|^{2}v_g=\limsup_{r\rightarrow\infty}\int_{\partial B_r}|\theta_{g}\wedge\alpha_{Z}|^{2}v_g=0.
\end{align}
L'égalité (\ref{mp6before}) et l'équation (\ref{mp6beforea}) donnent alors :
\begin{align}\label{mp6}
\int_{M}|\D^{D}\alpha|^2=0.
\end{align}
Nous déduisons des formules (\ref{mp5}) et (\ref{mp6}) la formule suivante : 
\begin{align}
\int_{M}\big(|D\alpha|^2 +\Ric^{D}(\alpha^{\sharp},\alpha^{\sharp})\big)=\omega_{n}Lm^{D}_{h}(Z).
\end{align}
Par conséquent, lorsque l'opérateur de Ricci est positif, la masse $m^{D}_{h}$ est une forme quadratique positive.
\end{proof}


\labelsep .5cm

\end{document}